\documentclass[12pt]{amsart}

\usepackage{amsfonts,verbatim,amsmath,amsthm,amssymb,amscd,enumerate,eucal,url,stmaryrd}

\usepackage[pagebackref,colorlinks=true]{hyperref}  

\numberwithin{equation}{section}
\usepackage{graphicx}
\newtheorem*{mainThrm*}{Theorem}
\newtheorem{thrm}{Theorem}[section]
\newtheorem{lemma}[thrm]{Lemma}
\newtheorem{prop}[thrm]{Proposition}
\newtheorem{cor}[thrm]{Corollary}

\newtheorem{rmrk}[thrm]{Remark}

\newtheorem*{conv*}{Conventions}

\newcommand{\lc}{\langle}
\newcommand{\rc}{\rangle}

\newcommand{\ind}[1]{\text{\scalebox{0.90}{\ensuremath #1}}}
\newcommand{\pow}[1]{\text{\scalebox{1.05}{\ensuremath #1}}}

\begin{document}

\begin{abstract}
We consider certain fiber bundles over a paraquaternionic contact manifolds, called twistor and reflector spaces, and show that these carry an intrinsic geometric structure that is always integrable.

\end{abstract}

\keywords{paraquaternionic, twistor, reflector, contact}

\subjclass{58G30, 53C17}
\title[Twistor and Reflector spaces for paraquaternionic contact manifolds]
{Twistor and Reflector spaces for paraquaternionic contact manifolds}
\date{\today}

 \author{Stefan Ivanov}
\address[Stefan Ivanov]{University of Sofia, Faculty of Mathematics and Informatics,
blvd. James Bourchier 5, 1164, Sofia, Bulgaria}
\address{and Institute of Mathematics and Informatics, Bulgarian Academy of
Sciences} \email{ivanovsp@fmi.uni-sofia.bg}

 \author{Ivan Minchev}
\address[Ivan Minchev]{University of Sofia, Faculty of Mathematics and Informatics,
blvd. James Bourchier 5, 1164, Sofia, Bulgaria}
 \email{minchev@fmi.uni-sofia.bg}
 
 \author{Marina Tchomakova}
\address[Marina Tchomakova]
{University of Sofia, Faculty of Mathematics and Informatics,
blvd. James Bourchier 5, 1164, Sofia, Bulgaria} \email{marina.chomakova@fmi.uni-sofia.bg}

\maketitle

\setcounter{tocdepth}{2}


\section{Introduction}

The geometry of paraquaternionic contact structures is essentially a tool to study some  special type of codimension three distributions on $(4n+3)$-manifolds with properties closely related to the algebra of the paraquaternions, known also as split quaternions \cite{Swann}, quaternions of the second kind \cite{L}, and complex product structures \cite{AS}. The paraquaternionic contact structures,  introduced in \cite{CIZ} may be considered as a  generalization of the para 3-Sasakian geometry developed in \cite{AK0,Swann}.  In many ways the paraquaternionic contact structures resemble the geometry of quaternionic contact manifolds, introduced by O. Biquard \cite{Biq}, which has been very useful in relation to the quaternionic contact Yamabe problem and the determination of the extremals and the best constant in a certain $L^2$ Folland-Stein inequality on the quaternionic Heisenberg group \cite{IMV1,IMV2,IP,IMV3,IV}. Despite the similarities between these two types of geometry there are also some major differences determined mainly by the fact that in the paraquaternionic contact setting one is often forced to consider sub-hyperbolic PDEs instead of subelliptic.

As shown in \cite{Biq}, the study of quaternionic contact structures leads back in a natural way to 
the study of a particular class of integrable CR manifolds (which are never pseudo-convex)---called twistor spaces---that appear as certain sphere bundles over the base quaternionic contact manifold  (see also \cite{DIM}). This is a generalization of the concept of a twistor space of a quaternionic K\"ahler manifold~\cite{Sal}. In the paraquaternionic contact case, we have two different types of bundles: the twistor space $\mathcal Z$ and the reflector space $\mathcal R$. The situation is very similar to the discussion in \cite{IMZ}. The fibers of $\mathcal Z$ are diffeomorphic to the 2-sheeted hyperboloid $x^2+y^2-z^2=-1$ in $\mathbb R^3$, whereas the fibers of $\mathcal R$ are diffeomorphic to the 1-sheeted hyperboloid $x^2+y^2-z^2=1$ (see Section~\ref{TandRsp} below for the details). The purpose of this paper is to show the following: 
\begin{mainThrm*} If $(M,H)$ is any paraquaternionic contact manifold with twistor space $\mathcal Z$ and reflector space $\mathcal R$ then we have a natural integrable CR structure on  $\mathcal Z$ and a natural integrable para-CR structure on $\mathcal R$. The Levi form for each of these structures is of signature $(2n+2,2n+2)$. 
\end{mainThrm*}

The proof of this theorem is divided into several steps throughout the paper and follows from the results obtained in propositions \ref{prop_twistor}, \ref{prop_reflector}, \ref{inv_K_J} and \ref{N_vanish}.

\begin{conv*}\label{conven}
 In the   paper we use the following general conventions:
\begin{enumerate}[a)]
\item The indices $s$ and $t$ are usually running from 1 to 3 (when nothing else specified).
\item The indices  $i,j,k$ always represent a   positive (cyclic) permutation of $1,2,3$.
\item The summation symbol $\sum_{(ijk)}$ indicates a summation over all positive permutations $(ijk)$ of $1,2,3$; that is 
\begin{equation*}
\sum_{(ijk)}X_{ijk}\ =\ X_{123}+X_{231}+X_{312}
\end{equation*}
\item We fix the signs: $\epsilon_1=-1$, $\epsilon_2=-1$ and $\epsilon_3=1$.
\end{enumerate}
\end{conv*}

\textbf{Acknowledgments}  

The research of S.I.  is partially supported by Contract KP-06-H72-1/05.12.2023 with the National Science Fund of Bulgaria,  by Contract 80-10-181/22.4.2024   with the Sofia University "St.Kl.Ohridski" and  the National Science Fund of Bulgaria, National Scientific Program ``VIHREN", Project KP-06-DV-7. The research of I.M.  is partially  financed by the European Union-Next Generation EU, through the National Recovery and Resilience Plan of the Republic of Bulgaria, project N:
BG-RRP-2.004-0008-C01. 
The research of M. Tch. is partially supported by Contract KP-06-H72-1/05.12.2023 with the National Science Fund of Bulgaria and  by Contract 80-10-181/22.4.2024  with the Sofia University "St.Kl.Ohridski".

\section{Preliminaries}

\subsection{The algebra of the split-quaternions}
Both the quaternions and the split-quaternions are real Clifford algebras generated by a two-dimensional non-degenerate quadratic form. In the negative-definite case we obtain the algebra of the quaternions, whereas in the other two cases---of a positive-definite or indefinite quadratic form---we get the same (up to an isomorphism) Clifford algebra, which is denoted here by $\mathbb B$ and is called the algebra of the split-quaternions (or paraquaternions). The elements of  $\mathbb B$ are generally represented in the form 
\begin{equation*}
a\ =\ a_\ind{0}+a_\ind{1}\, j_\ind{1}+a_\ind{2}\, j_\ind{2}+a_\ind{3}\, j_\ind{3},
\end{equation*}
where $a_\ind{s}$ are real numbers and $j_\ind{s}$ are basic split-quaternions; that is, some fixed elements of $\mathbb B$ satisfying the identities
\begin{equation*}
j_\ind{1}^\pow{2}=j_\ind{2}^\pow{2}=1,\qquad j_\ind{1}\,j_\ind{2}=-j_\ind{2}\,j_\ind{1}=j_\ind{3}.
\end{equation*}
The remaining multiplication rules for $\mathbb B$ are easily derived from these:  
$$
j_\ind{3}^\pow{2}=-1,\qquad j_\ind{2}\,j_\ind{3}=-j_\ind{3}\,j_\ind{2}=-j_\ind{1},\qquad  j_\ind{3}\,j_\ind{1}=-j_\ind{1}\,j_\ind{3}=-j_\ind{2}.
$$
The conjugate to $a$ is defined by
$
\overline {a}\ =\ a_\ind{0}-a_\ind{1}\, j_\ind{1}-a_\ind{2}\, j_\ind{2}-a_\ind{3}\, j_\ind{3}.
$
We have the typical identity $\overline{a\,b}=\overline{b}\,\overline{a}$.
The real and imaginary parts of a split-quaternion are given by $Re(a)=a_\ind{0}$ and $Im(a)=a_\ind{1}\, j_\ind{1}+a_\ind{2}\, j_\ind{2}+a_\ind{3}\, j_\ind{3}$. There is a natural inner product on $Im(\mathbb B)=\mathbb R^3$,  
\begin{equation}\label{inn_prod_r3}
\lc a,b \rc\ = - \ Re(a\,b)\ =\ -a_1\,b_1-a_2\,b_2+a_3\,b_3,
\end{equation}
and a cross product "$\times$",
\begin{equation}\label{cross_prod_r3}
a\times b\ =\ \sum_{\substack {s,t=1\\ s\ne t}}^3a_s\,j_s\,b_t\,j_t,
\end{equation}
so that 
\begin{equation*}
\lc a\times b\,,\,c\rc \ =\ \det \begin{pmatrix}
a_1 & b_1 & c_1\\
a_2 & b_2 & c_2\\
a_3 & b_3 & c_3
\end{pmatrix} ,\qquad a,b,c\in Im(\mathbb B).
\end{equation*}

Observe that $\mathbb B$ is isomorphic to the algebra $M_2(\mathbb R)$ of all $2\times2$ matrices with real entries under the identification
\begin{equation*}
j_1\ =\ \begin{pmatrix}
\ 1 &\ \, 0\,\\
\ 0 &  -1\,
\end{pmatrix},
\qquad j_2\ =\ \begin{pmatrix}
\ 0 &\ 1\,\\
\ 1 &\  0\,
\end{pmatrix},
\qquad j_3\ =\ \begin{pmatrix}
\ \,0 &\ 1\,\\
-1 &\  0\,
\end{pmatrix}.
\end{equation*}

\vspace{0.3cm}
Let $SO(1,2)$  be the group of all $3\times 3$ real matrices of determinant $1$ that preserve the inner product~ \eqref{inn_prod_r3}. We shall need the following basic lemma, which is easily derived from the multiplication rules of $\mathbb B$.

\begin{lemma}\label{lemma-basic-q} Three split-quaternions $\gamma_\ind{1},\gamma_\ind{2}$ and $\gamma_\ind{3}$ satisfy the identities
\begin{equation}\label{mult_rules_gamma}
\gamma_\ind{1}^\pow{2}=\gamma_\ind{2}^\pow{2}=1,\qquad \gamma_\ind{1}\,\gamma_\ind{2}=-\gamma_\ind{2}\,\gamma_\ind{1}=\gamma_\ind{3}
\end{equation} if and only if there exists a matrix $A=(a_{st})\in SO(1,2)$ so that $\gamma_s=\sum_t a_{st}\,j_t$,  $s=1,2,3.$

\end{lemma}

If we regard the vector space $\mathbb B^n$ (the elements of $\mathbb B^n$ are thought of as column vectors) as a right $\mathbb B$-module, the multiplication from the left with
$n\times n$ matrices with entries in $\mathbb B$ represents the space of all $\mathbb B$-linear endomorphisms of $\mathbb B^n$. We define $Sp(n,\mathbb B)$ to be the group of all $\mathbb B$-linear transformations that preserve the inner product $\lc x,y\rc=Re(\overline{x}^T\,y)$, $x,y\in \mathbb B^n,$
\begin{equation*}
Sp(n,\mathbb B)=\Big\{A\in M_n(\mathbb B)\ :\ \overline{A}^TA=1\Big\}.
\end{equation*}
In particular,  $Sp(1,\mathbb B)$ is the group of the unite split-quaternions,
\begin{equation*}
Sp(1,\mathbb B)=\Big\{z=z_\ind{0}+z_\ind{1}\, j_\ind{1}+z_\ind{2}\, j_\ind{2}+z_\ind{3}\, j_\ind{3}\ :\ z_\ind{0}^\pow{2}-z_\ind{1}^\pow{2}-z_\ind{2}^\pow{2}+z_\ind{3}^\pow{2}=1\Big\}.
\end{equation*}

Consider the action of the direct product $Sp(n,\mathbb B)\times Sp(1,\mathbb B)$ on the vector space $\mathbb B^n$, defined by
\begin{equation}\label{action}
(A,z)\cdot x\ =\ A\,x\,\overline{z},
\end{equation}
and let us fix (once and for all) an identification $\mathbb B^n=\mathbb R^{4n}$. Since the induced inner product is of signature $(2n,2n)$, we obtain an embedding of the quotient group
\begin{equation*}
\frac{Sp(n,\mathbb B)\times Sp(1,\mathbb B)}{\big\{\pm (1,1)\big\}}
\end{equation*}
into the matrix group $SO(2n,2n)$. The image of this embedding is denoted by $Sp(n,\mathbb B) Sp(1,\mathbb B)$ and consists of all elements of $SO(2n,2n)$ that preserve the three dimensional subspace $\mathcal Q\subset End(\mathbb R^{4n})$ generated by the right action of $Im(\mathbb B)$ on $\mathbb B^n$.

\subsection{Paraquaternionic contact structures}\label{subsection-pqc-str}
Consider a $4n$-dimensional smooth distribution $H$ on a $(4n+3)$-dimensional manifold $M$. Suppose that at each point $p$ in an open subset $U\subset M$, we are given a triple $(\eta_1,\eta_2,\eta_3)$ of 1-forms on $T_pM$, a triple $(I_1,I_2,I_3)$ of endomorphisms of $H_p\subset T_pM$, and a non-degenerate quadratic form  $g$ on $H_p$, all depending smoothly on the point $p$. The list $(\eta_s,I_s,g)$ is called a (local) paraquaternionic contact (shortly: pqc) structure  for $H$ on $U$, if the following three conditions are satisfied at each $p\in U$:

\par i) $H_p=\Big\{A\in T_pM\ :\ \eta_1(A)=\eta_2(A)=\eta_3(A)=0\Big\}$;

\par ii)  $d\eta_s(X,Y)=2g(I_sX,Y),\quad \forall X,Y\in H_p$, \quad s=1,2,3;

\vspace{0.1cm}
\par iii)  $I_\ind{1}^\pow{2}=I_\ind{2}^\pow{2}=\text{id},\quad I_\ind{1}I_\ind{2}=-I_\ind{2}I_\ind{1}=I_\ind{3}.$

\vspace{0.3cm}

 Clearly, for every local pqc-structure $(\eta_s,I_s,g)$ for $H$,  the quadratic form $g$ must be of signature (2n,2n). The pair $(M,H)$ is called a paraquaternionic contact manifold if around each point of $M$ there exists at least one local pqc-structure for $H$. Here arises the natural question: To what extent are the different local pqc-structures determined by the distribution $H$? The answer is given by the following

\begin{lemma} \label{lemma-pqc} Suppose that $(M,H)$ is a pqc manifold.
If $(\eta_s,I_s,g)$ and  $(\eta'_s,I'_s,g')$ are two pqc-structures for $H$ on an open set $U\subset M$, then 
$$(\eta'_1,\eta'_2,\eta'_3)=f(\eta_1,\eta_2,\eta_3)\mathcal
S,\quad (I'_1,I'_2,I'_3)=(I_1,I_2,I_3)\mathcal S,\quad g'= f\,g,$$
for some non-vanishing real valued smooth function $f$ on $U$  and some matrix-valued smooth function $\mathcal S=(a_{ij}):U\rightarrow SO(1,2)$. 
\end{lemma}

\begin{proof}
By assumption $H=\cap_{s=1}^3Ker(\eta_s)=\cap_{s=1}^3Ker(\eta'_s)$ and therefore
 there
exists a matrix-valued function  $\mathcal A=(a_{st}):U\rightarrow
GL(3)$ so that $\eta'_s=\sum_{t=1}^3 a_{st}\eta_t$, $s=1,2,3.$ Applying the
exterior derivative to both sides of this equation and taking the restriction of the resulting 2-forms to the distribution $H$, we obtain
\begin{equation}\label{ad-lemma1}
d\eta'_s|_H=\sum_t a_{st}(d\eta_t|_H).
\end{equation}

If $G'$ is a field of endomorphisms of $H$ defined by the equation $g'(X,Y)=g(G'X,Y)$, $X,Y\in H$, then
$d\eta'_s(X,Y)=g(G'I'_sX,Y)$ and by \ref{ad-lemma1}, 
\begin{equation*}
G'I'_s=\sum_t a_{st}I_t.
\end{equation*}
This yields 
\begin{equation*}
I'_\ind{1}=(I'_\ind{2})^\pow{-1}\,I'_\ind{3}=(G'I'_\ind{2})^\pow{-1}\,(G'I'_\ind{3})=(\sum_s a_{2s}I_s)^\pow{-1}(\sum_t a_{3t}I_t)\ \in\ \text{span}_{\mathbb R}\ \{id_H,I_1,I_2,I_3\},
\end{equation*}
and similarly for $I'_2$ and $I'_3$. Let us observe that $\text{span}_{\mathbb R}\ \{id_H,I_1,I_2,I_3\}\subset
End(H)$ is an algebra with respect to the usual composition of
endomorphisms, which is isomorphic to the algebra of the
split-quaternions. 
Therefore, by Lemma~\ref{lemma-basic-q}, we have
$$\text{span}_{\mathbb R}\ \{I_1,I_2,I_3\}\ =\ \text{span}_{\mathbb R}\ \{I'_1,I'_2,I'_3\}.$$
In particular, this yields that $I'_1,I'_2,I'_3$ are skew-symmetric with respect to both $g$ and $g'$.
Furthermore, we calculate
\begin{gather*}
g\Big(\big(G'I'_1I'_2+I'_2G'I'_1\big)X,Y\Big)=g(G'I'_3X,Y)-g(G'I'_1X,I'_2Y)\\
=g'(I'_3X,Y)-g'(I'_1X,I'_2Y)=0,
\end{gather*}
i.e., $G'I'_1$ anti-commutes with $I'_2$ and similarly also with $I'_3$. Therefore $G'I'_1$ must be proportional to $I'_1$, i.e., $G'$  is proportional to  the identity; this means,  
$$g'=f\, g $$ for some appropriate non-vanishing real-valued function $f$. The rest follows from Lemma~\ref{lemma-basic-q}.
\end{proof}

An important consequence of the above lemma is that to each pqc manifold $(M,H)$ we can associate a canonical line bundle $\mathcal G(M)\rightarrow M$, so that if  $(\eta_s,I_s,g)$ is a local pqc structure for $H$, then $g$ is a local section of  $\mathcal G(M)$. Furthermore, the vector bundle $\pi:\mathcal Q(M)\rightarrow M$ with fiber (over $p$) 
\begin{equation}\label{def_Q}
\mathcal Q_p=\text{span}\{I_1,I_2,I_3\},
\end{equation}
is also globally defined. It has a canonical inner product,
\begin{equation}\label{in_prod_Q}
\lc I_s,I_t\rc=
\begin{cases}
\epsilon_s,\ \text{if}\ s=t\\
0,\ \text{otherwise}
\end{cases},
\qquad \epsilon_1=\epsilon_2=-\epsilon_3=-1
\end{equation}
of signature $(-,-,+)$ and an orientation (defined by the ordering of  $I_1,I_2$ and $I_3$).


\subsection{Invariant tensor decomposition}\label{invariant_decomposition} Let $(M,H)$ be a pqc manifold and consider some local pqc-structure $(\eta_s,I_s,g)$ for $H$, defined around a fixed  $p\in M$. Each endomorphism $\Psi\in End(H_p)$ can be decomposed uniquely into a sum of four components,
 $\Psi=\Psi^{+++}+\Psi^{+--}+\Psi^{-+-}+\Psi^{--+}$, where $\Psi^{+++}$ commutes with $I_1$, $I_2$ and $I_3$, $\Psi^{+--}$ commutes with $I_1$, and anti-commutes with $I_2$ and $I_3$,  and etc.
Explicitly,
\[
\begin{split}
4\Psi^{+++}=\Psi+I_1\Psi I_1+I_2\Psi I_2-I_3\Psi I_3;\quad
4\Psi^{+--}=\Psi+I_1\Psi I_1-I_2\Psi I_2+I_3\Psi I_3;\\
4\Psi^{-+-}=\Psi-I_1\Psi I_1+I_2\Psi I_2+I_3\Psi I_3;\quad
4\Psi^{--+}=\Psi-I_1\Psi I_1-I_2\Psi J_2-I_3\Psi I_3.
\end{split} 
\]

Clearly, this decomposition depends on the particular choice of a pqc-structure. To obtain an invariant decomposition we shall  consider the action of the Casimir operator $\dagger$ on $End(H_p)$, given by 
\[\dagger(\Psi)=I_1\Psi I_1+I_2\Psi I_2-I_3\Psi I_3.
\] 
The leading signs $(+,+, -)$
 in the above summation are opposite to the signature of the invariant inner product on $\mathcal Q_p$ (cf. \eqref{in_prod_Q}) and therefore $\dagger$ must be invariant too. It is easily seen that this Casimir operator has eigenvalues $3$ and  $-1$, and that, if $\Psi=\Psi_{[3]}+\Psi_{[-1]}$ is the induced decomposition of $\Psi\in End(H_p)$ into a sum of eigenvectors, then  
\[
\begin{split}
\Psi_{[3]}= \Psi^{+++}\qquad \text{and}\qquad
\Psi_{[-1]} =\Psi^{+--}+\Psi^{-+-}+\Psi^{--+}.
\end{split}
\]
\subsection{The canonical connection} In general, a pqc manifold $(M,H)$ is a parabolic type of geometry that can not be characterized by a linear connection on the tangent bundle of $M$; it requires a more complicated construction involving a certain Cartan connection, which we shall not deal with here. Instead, we shall use an auxiliary assumption; we require that the naturally induced line bundle $\mathcal G(M)\rightarrow M$  (cf. \S~\ref{subsection-pqc-str}) admits a global non-vanishing section $g$; that is, there is a globally defined $g$ on $M$ so that around each point, one can find at least one local pqc-structure for $H$ of the form  $(\eta_s,I_s,g)$  (with last entry the same $g$). 

The triple $(M,H,g)$ is already a much simpler type of geometry that can be characterized by a unique linear connection $\nabla$ on the tangent bundle of $M$ (as shown in \cite{CIZ}) called the canonical connection of the triple. We shall summarize all the relevant properties of this connection below. Let us first observe that the differential invariants produced by $\nabla$ depend strongly on the choice of $g$.  If we are interested only in the geometry defined by $(M,H)$, we need to consider those differential invariants that remain unchanged after an arbitrary multiplication of $g$ by a non-vanishing function (cf. Lemma~\ref{lemma-pqc}). The relationship between $(M,H,g)$ and $(M,H)$ is  similar to that between the Riemannian and the conformal Riemannian geometry.

In \cite{CIZ} is shown (with a slightly different notation) that if the dimension of $M$ is at least 11, to each choice of (a global) $g$, there exists a unique complementary (vertical) distribution  $V\subset TM$ on $M$, 
\begin{equation}\label{TM_split}
TM=H\oplus V.
\end{equation} 
If we pick any local pqc-structure $(\eta_s,I_s,g)$ for $H$, then
 $V$ is the real span of local vector fields $\xi_1,\xi_2$ and $\xi_3$ on $M$---called Reeb vector fields---that are defined by the equations:
\begin{equation}\label{Reeb}
\begin{aligned}
&(i) \quad\eta_s(\xi_t)=
\begin{cases}
\epsilon_s,\ \text{if}\ s=t \\
0,\ \text{otherwise}
\end{cases},
\qquad \epsilon_1=\epsilon_2=-\epsilon_3=-1;\\
&(ii)\quad d\eta_s(\xi_t,X)+d\eta_t(\xi_s,X)=0,\quad \forall X \in H, \quad  s,t=1,2,3.
\end{aligned} 
\end{equation}

\begin{rmrk}
In (the lowest) dimension $7$ the existence of Reeb vector fields is an additional condition on the structure, which we shall assume here is always satisfied. 
\end{rmrk}
At each $p\in M$, the vector space $H_p$ is isomorphic as a  $Sp(n,\mathbb B)Sp(1,\mathbb B)$-module to $\mathbb B^n$ \big(with the action \eqref{action}\big) and the set of all isomorphisms from $H_p$ to $\mathbb B^n$ constitutes a fiber over $p$ of a certain principle bundle $\mathcal P(M)\rightarrow M$ with a structure group $Sp(n,\mathbb B)Sp(1,\mathbb B).$ The Reeb vector fields \eqref{Reeb} allow us to extend the action of $Sp(n,\mathbb B)Sp(1,\mathbb B)$ on $H_p$ to an action on the whole tangent space at  $p$, $T_pM=H_p\oplus V_p$, by declaring that $Sp(n,\mathbb B)Sp(1,\mathbb B)$ acts on the Reeb vector fields $\xi_s$ in the same way as it acts on the endomorphisms $I_s\in End(H_p)$. It is easily verified (using Lemma~\ref{lemma-pqc}) that this action remains unchanged if we replace the initial pqc-structure $(\eta_s,I_s,g)$ with another (of course, the  Rieb vector fields must undergo a respective transformation as well) as long as the $g$-entry remains the same; that is, the choice of $g$ allows us to consider  $T_pM$ as  a $Sp(n,\mathbb B)Sp(1,\mathbb B)$-module isomorphic to $\mathbb B^n\oplus Im(\mathbb B)$ and the set of all isomorphisms is a principle fiber bundle $\mathcal P(M)\rightarrow M$ with a structure group $Sp(n,\mathbb B)Sp(1,\mathbb B)$.

The canonical connection $\nabla$ is a principle $Sp(n,\mathbb B)Sp(1,\mathbb B)$-connection on $\mathcal P$, whose torsion tensor 
\begin{equation}\label{def_tor}
T(A,B)=\nabla_AB-\nabla_BA-[A,B],\qquad A,B\in TM,
\end{equation} 
can be described as follows. 

We define three (local) 2-forms $\omega_1$, $\omega_2$ and  $\omega_3$ on $M$ by setting
\begin{equation}\label{def_omegas}
\omega_s(A,B)\ =\ g\big(I_s (A_H),B_H\big),\qquad s=1,2,3,
\end{equation}
where by subscript $H$ we mean projection onto $H$ w.r.t. the decomposition \eqref{TM_split}.  There exists a (unique) triple $(Scal,\tau,\mu)$, where $Scal$ is a (global) function on $M$;  $\tau$ and $\mu$ are globally defined  traceless symmetric sections of the endomorphism bundle $End(H)\rightarrow M$,  satisfying $\tau=\tau_{[-1]}$, $\mu=\mu_{[3]}$ (cf. \S~\ref{invariant_decomposition}), so that: 
\vspace{0.2cm}
\begin{equation}\label{torsion}
\begin{split}
&(i)\quad T(X,Y)\ =\ -2\omega_1(X,Y)\xi_1 - 2\omega_2(X,Y)\xi_2 + 2\omega_3(X,Y)\xi_3\\
&(ii)\quad T(\xi_s,X)\ =\ \Big(\frac{1}{4}\,\big(I_s\,\tau \, -\,  \tau\,I_s\big) \ + \ I_s\,\mu\Big)X \\
&(iii)\quad T(\xi_s,\xi_t)\ =\ -\frac{Scal}{8n(n+1)} \, \xi_s\times\xi_t - \big[\xi_s,\xi_t\big]_H\qquad cf.\ \eqref{cross_prd}\\
\end{split}
\end{equation}   
\noindent  for any $X,Y\in H,\  s,t=1,2,3$. Notice also that the vertical distribution $V$ has an induced inner product $\lc,\rc$ of signature $(-,-,+)$, so that
\begin{equation}\label{inner_pr_V}
\lc\xi_s,\xi_t\rc=
\begin{cases}
\epsilon_s,\ \text{if}\ s=t\\
0,\ \text{otherwise}.
\end{cases}
\end{equation}
On  $V$ we have also a natural orientation and a cross product "$\times$" (cf. \eqref{cross_prod_r3}),
\begin{equation}\label{cross_prd}
\xi_1\times\xi_2=\xi_3,\qquad\xi_2\times\xi_3=-\xi_1,\qquad \xi_3\times\xi_1=-\xi_2.
\end{equation}
Clearly, the two vector bundles $V\rightarrow M$ and $\pi:\mathcal Q(M)\rightarrow M$ over $M$ (cf. \eqref{def_Q}) are isomorphic and the cross product on $V_p$ corresponds to the half-commutator on $\mathcal Q_p$, 
\begin{equation*}
\frac{1}{2}\big[I,\,J\big]\ =\ \frac{1}{2}\big(I\,J\, -\, J\,I\big),\qquad I,J\in \mathcal Q_p.
\end{equation*}

Since both $\lc,\rc$ and "$\times$" are $\nabla$-parallel, locally, on the same domain where the considered local pqc-structure $(\eta_s,I_s,g)$ is defined, we can find certain $1$-forms $\alpha_1,\alpha_2$ and $\alpha_3$ (called connection 1-forms) so that 
\begin{equation}\label{nabla_cross_1}
\nabla_A I_s\ =\  \frac{1}{2}\Big[\sum_t\alpha_t(A)I_t,\ I_s\Big],
\end{equation}
or equivalently
\begin{equation}\label{nabla_cross_2}
\nabla_A \xi_s\ =\  \Big(\sum_t\alpha_t(A)\xi_t\Big)\times \xi_s,
\end{equation}
for all $A \in T_pM$ and $s=1,2,3$. As shown in \cite{CIZ}, the connection 1-forms are completely determined  by the exterior derivatives of the  three 1-forms $\eta_s$ and the function $Scal$, 
\begin{equation}\label{coneforms}
\begin{aligned}
&\alpha_i(X)=d\eta_k(\xi_j,X)=-d\eta_j(\xi_k,X),\\
&\alpha_i(\xi_s)=d\eta_s(\xi_j,\xi_k)-\delta_{is}\Bigg(\frac{Scal}{
16n(n+2)}\\
&\hspace{3cm}+\frac12\left(d\eta_1(\xi_2,\xi_3)+d\eta_2(\xi_3,\xi_1)+d\eta_3(\xi_1,\xi_2)\right)\Bigg), 
\end{aligned}
\end{equation}
\noindent  for all $X\in H$ and $s=1,2,3$; where $\delta_{is}$ is the Kronecker delta and $(ijk)$ is any positive permutation of $1,2,3$.

\subsection{Curvature}
It turns out that not only the torsion (cf. \eqref{torsion}), but also many of the contractions of the curvature tensor 
\begin{equation*}
R(A,B)=\big[\nabla_A,\nabla_B\big]\ -\ \nabla_{[A,B]},\qquad A,B\in TM,
\end{equation*} 
are completely determined by the triple $(Scal,\tau,\mu)$. Consider a local frame $e_a\in H,$ $1\le a\le 4n$ for $H$ and let $e^*_a\in H$ be its dual; that is, the frame defined by the equations
\begin{equation}\label{dual_basis}
g(e_a,e^*_b)\ =\ 
\begin{cases}
1,\ \ \text{if}\ a=b\\
0,\ \ \text{otherwise}
\end{cases},\quad
\forall a,b=1,\dots,4n.
\end{equation}
The Ricci curvature, $Ric$, is defined by
\begin{equation*}
Ric(A,B)=\sum_a g(R(e_a,A)B,e^*_a),\qquad  A,B\in TM.
\end{equation*}
According to \cite{CIZ}, we have 
\begin{equation*}
Ric(X,Y)\ =\ g\Big(\frac{Scal}{4n}X\ +\ (2n+2)\tau(X)\ +\ (4n+10)\mu (X),\,Y\Big),
\end{equation*}
for all $X,Y\in H.$ In particular, 
$
Scal=\sum_a Ric(e_a,e^*_a),
$
i.e., $Scal$ is indeed the scalar curvature of $\nabla$.

Since, by design, $\nabla$ is a principle $Sp(n,\mathbb B)Sp(1,\mathbb B)$ connection, its curvature splits into a sum of two components, $sp(n,\mathbb B)\oplus sp(1,\mathbb B)$. We shall use the Ricci 2-forms $\rho_s$ to represent the $sp(1,\mathbb B)$ component of the curvature: 
\begin{equation}\label{curv-Ricci}
\big[R(A,B),I_s\big]\ = \sum_{t=1}^3 \rho_t(A,B)\big[I_t,I_s\big],\qquad \forall A,B\in TM,\quad s=1,2,3,
\end{equation}
or equivalently
\begin{equation*}
\rho_s(A,B)\ =\ \frac{\epsilon_s}{4n}\sum_a g\big(R(A,B)e_a,I_se^*_a\big).
\end{equation*}

By a result in \cite{CIZ},
\begin{equation}\label{expr_for_rho}
\begin{aligned}
&\rho_s(X,Y)\ =\ -\ \epsilon_s\,g\bigg(\frac{1}{2}\big(\tau\,I_s+I_s\,\tau\big)X +2\Big(\mu+\frac{Scal}{16n(n+2)}\Big) I_sX,\,Y\bigg)\\
&\rho_i(X,\xi_i)\ =\  -\ \frac{d\, Scal(X)}{32n(n+2)}\ +\ \frac{1}{2}g\Big(I_i\big[\xi_j,\xi_k\big]_H + I_j\big[\xi_k,\xi_i\big]_H+I_k\big[\xi_i,\xi_j\big]_H ,\, X\Big)\\
&\rho_i(X,\xi_s)\ =\ g\Big(I_s\big[\xi_j,\xi_k\big]_H,\,X\Big), \qquad i\ne s,
\end{aligned}
\end{equation}
for all $X,Y\in H$ and $s=1,2,3$; where $(ijk)$ is any positive permutation of 1,2,3 ($d\,Scal$ is the differential of $Scal$). For the values of the three Ricci 2-forms on a pair of vertical vector fields, we have the identity
\begin{equation}
\rho_i(\xi_i,\xi_j)+\rho_k(\xi_k,\xi_j)\ =\ \frac{d\,Scal(\xi_j)}{16n(n+2)}
\end{equation}

\section{Twistor and reflector spaces} \label{TandRsp}
The twistor space $\mathcal Z$ and the reflector space $\mathcal R$ of a pqc manifold $(M,H)$ are defined as subbundles of the canonical vector bundle $\pi:\mathcal Q(M)\rightarrow M$  (cf. \eqref{def_Q}). The corresponding fibers over a point $p\in M$ are 
\begin{equation*}
\mathcal Z_p\ =\ \Big\{I\in \mathcal Q_p(M)\ :\ I^2=-\text{id}\Big\}\qquad\text{and}\qquad 
\mathcal R_p\ =\ \Big\{I\in \mathcal Q_p(M)\ :\ I^2=\text{id}\Big\}.
\end{equation*}
The purpose of this section is to prove the following two propositions.

\begin{prop}\label{prop_twistor} On the twistor space $\mathcal Z$, there exists a natural codimension 1 distribution $\mathcal K\subset T\mathcal Z$ and a smooth field $J$ of endomorphisms of $\mathcal K$ that satisfies $J^2=-\text{id}$ \big(such a pair $(\mathcal K,J)$ is called an almost CR structure\big). 

Furthermore, if $\eta$ is any local 1-form on $\mathcal Z$ with $\mathcal K=\ker(\eta)$, then at each $I\in\mathcal Z$, $d\eta(J.,.)$ is a non-degenerate symmetric 2-tensor on $\mathcal K_I$ of signature $(2n+2,2n+2)$\footnote{$\dim(M)=4n+3$}; that is, the Levi form of the almost CR structure on $\mathcal Z$ is of signature $(2n+2,2n+2)$.
\end{prop}

\begin{prop} \label{prop_reflector} On the reflector space $\mathcal R$, there exists a natural codimension 1 distribution $\mathcal K\subset T\mathcal R$ and a smooth field $J$ of endomorphisms of $\mathcal K$ that satisfies $J^2=\text{id}$ \big(such a pair $(\mathcal K,J)$ is called an almost para-CR structure\big). 

Furthermore, if $\eta$ is any local 1-form on $\mathcal R$ with $\mathcal K=\ker(\eta)$, then at each $I\in\mathcal R$, $d\eta(J.,.)$ is a non-degenerate symmetric 2-tensor on $\mathcal K_I$ of signature $(2n+2,2n+2)$; that is, the Levi form of the almost para-CR structure is of signature $(2n+2,2n+2)$.
\end{prop}

  Later on (Section~\ref{integrability}) we will show that both the almost CR structure on $\mathcal Z$ and the almost para-CR structure on $\mathcal R$ are in fact integrable.

\subsection{The induced structure on $\mathcal Q$} \label{sec_Q}
To begin with, let us fix an arbitrary non-vanishing section $g$ of the line bundle $\mathcal G(M)\rightarrow M$ (cf. \S~\ref{subsection-pqc-str}) and consider the corresponding canonical connection  $\nabla$ on $TM$. We shall use $\nabla$ to induce a certain structure on the tangent space of the vector bundle $\mathcal Q=\mathcal Q(M)$.  Indeed, since $\nabla$ preserves the vector bundle $\mathcal Q\subset End(TM)$, it defines a horizontal distribution $\mathcal D\subset T\mathcal Q$ so that the horizontal lift $A^h$ (w.r.t. $\nabla$) of any vector field $A$ on $M$ is a vector field on $\mathcal Q$ tangent to $\mathcal D$. On the other hand, there is a distribution $\mathcal F=\ker(\pi_\ast)\subset T\mathcal Q$ that consists of all vectors that are tangent to the fibers of the bundle $\pi:\mathcal Q\rightarrow M$. We have the direct sum decomposition
\begin{equation*}
T\mathcal Q\ =\ \mathcal D\ \oplus\ \mathcal F.
\end{equation*}
The differential $\pi_{\ast}$ of the projection map $\pi:\mathcal Q\rightarrow M$ at any $I\in \mathcal Q$ is an isomorphism between $\mathcal D_I$ and $T_pM$, where $p=\pi(I)$. There is also a natural isomorphism $\mathcal F_I\cong\mathcal Q_p$ that identifies the tangent vector  to a curve $t\mapsto I(t)\in \mathcal Q_p$  at $I(0)=I$ (that is, any element of  $\mathcal F_I$) with the respective derivative $\frac{d\,I(t)}{dt}_{|t=0}$ (which is as an element of the fiber $\mathcal Q_p$). 

Let us consider a (small enough) domain $U$ of local coordinates $u_\alpha$, $1\le \alpha\le 4n+3$ on $M$. For each $I\in \pi^{-1}(U)\subset \mathcal Q$, we have that $I=x_1I_1+x_2I_2+x_3I_3$, and thus we may consider the functions
\begin{equation}\label{loc_coord}
 u_\alpha{\,\text{\scalebox{0.70}{\ensuremath \circ}}\,} \pi,\,x_1,\,x_2,\,x_3,\qquad 1\le \alpha\le 4n+3,
\end{equation}
as local coordinates on $\mathcal Q$ (we shall abbreviate $u_\alpha{\,\text{\scalebox{0.70}{\ensuremath \circ}}\,}\pi$ to $u_\alpha$).
In this coordinate chart, the isomorphism between $\mathcal F_I$ and $\mathcal Q_p$ identifies $\frac{\partial}{\partial x_s}$ with $I_s$ for $s=1,2,3$. 
\begin{lemma}\label{lemma_lift}
Within the coordinate chart $\eqref{loc_coord}$, the horizontal lift $A^h$ of a vector field 
\begin{equation*} 
A\ =\ \sum_{a=1}^{4n+3}A_s\frac{\partial}{\partial u_a}
\end{equation*}
on $M$, at $I=\sum_s x_sI_s\in \mathcal Q$, is given by
\begin{equation}\label{lift}
\begin{aligned}
A_I^h\ &=\ \sum_{\alpha=1}^{4n+3}A_\alpha\frac{\partial}{\partial u_\alpha}\ -\ \sum_{s,t=1}^3x_s\Big\lc\nabla_AI_s,\epsilon_tI_t\Big\rc\frac{\partial}{\partial x_t} \\
&\qquad=\ \sum_{\alpha=1}^{4n+3}A_\alpha\frac{\partial}{\partial u_\alpha}\ +\ \sum_{(ijk)\footnotemark} \epsilon_i\, \Big(x_j\,\alpha_k(A)-x_k\alpha_j(A)\Big) \frac{\partial}{\partial x_i},
\end{aligned}
\end{equation}
\footnotetext{This indicates a sum over all positive permutations $(ijk)$ of $1,2,3$}
where $A^h_I$ denotes the value of $A^h$ at $I$, and $\alpha_s$ are the connection 1-forms of $\nabla$ $($cf. \eqref{nabla_cross_1}$).$
\end{lemma}
\begin{proof}  Consider a curve $t\mapsto \Big(u_\alpha(t),\,x_s(t)\Big)$ within the coordinate chart $\eqref{loc_coord}$, passing through a fixed $I\in\mathcal Q$ at a time $t=0$. Suppose that the tangent vector to this curve at $t=0$ is $A^h_I$. Then,
\begin{equation*}
0\ =\ \nabla_{A}\Big(\sum_s x_s(t){I_s}\Big)\ = \ \sum_s\Big(\dot x_s(0)I_s\ +\ x_s(0)\nabla_AI_s\Big)
\end{equation*}
and therefore, since $x_s(0)I_s=I$, 
\begin{equation}\label{dot_xs_0}
\dot x_s(0)=-\sum_t\,x_t\Big\lc\nabla_AI_t,\epsilon_sI_s\Big\rc;
\end{equation}
that is, for the horizontal lift $A^h_I$ we have
\begin{equation*}
A_I^h\ =\  \sum_{\alpha=1}^{4n+3}A_\alpha\frac{\partial}{\partial u_\alpha}\ +\ \sum_{s=1}^3\dot x_s(0)\frac{\partial}{\partial x_s},
\end{equation*}
where $\dot x_s(0)$ are given by \eqref{dot_xs_0}. Applying \eqref{nabla_cross_1} to the latter yields the result.
\end{proof}

\begin{lemma} For any two vector fields $A$ and $B$ on $M$, within a coordinate chart like \eqref{loc_coord},  the commutator of their respective horizontal lifts $A^h$ and $B^h$ at  any $I=\sum_s x_sI_s\in \mathcal Q$, is given by
\begin{equation*}
\big[A^h,B^h\big]_I\ = \ \big[A,B\big]^h_I\ +\ \sum_{(ijk)} 2\epsilon_i\Big(x_j\rho_k(A,B)-x_k\rho_j(A,B)\Big)\frac{\partial}{\partial x_i},
\end{equation*} 
where $\rho_s$ are the corresponding Ricci 2-forms (cf. \eqref{curv-Ricci}).
\end{lemma}
\begin{proof}
Using \eqref{lift}, we calculate
\begin{equation*}
\begin{aligned}
\big[A^h,B^h\big]_I\ &=
\ \sum_{\alpha=1}^{4n+3}\big[A,B\big]_\alpha\,\frac{\partial}{\partial u_\alpha}\ -\ \sum_{s,t=1}^3x_s\Big\lc\nabla_A\big(\nabla_BI_s\big)-\nabla_B\big(\nabla_AI_s\big),\,\epsilon_tI_t\Big\rc\frac{\partial}{\partial x_t}\\
 &\qquad =\big[A,B\big]^h_I\ -\ \sum_{s,t=1}^3x_s\Big\lc\big[R(A,B),I_s\big],\,\epsilon_tI_t\Big\rc\frac{\partial}{\partial x_t}
\end{aligned}
\end{equation*}
The result follows from \eqref{curv-Ricci}.
\end{proof}

Next, we consider two naturally defined (global) vector fields $\chi$ and $\mathcal N$ on $\mathcal Q$. At any $I= \sum_s x_sI_s\in \mathcal Q$, we set, with respect to the coordinate  chart~\eqref{loc_coord}, 
\begin{equation}\label{def_chi_N}
\chi\ =\ \sum_s  x_s\xi^h_s\qquad\text{and}\qquad \mathcal N\ =\  \sum_s  x_s\frac{\partial}{\partial x_s}.
\end{equation}
Clearly, $\mathcal N$ is a section of the vertical distribution $\mathcal F\subset T\mathcal Q$. On the other hand, the splitting  $TM=H\oplus V$ (cf. \eqref{TM_split}) defines a splitting of the horizontal distribution, $\mathcal D=\mathcal H\oplus \mathcal V$, and the vector field $\chi$ is everywhere tangent to  $\mathcal V$.

Suppose that $I\in \mathcal Q$, considered as an endomorphism of the vector space $H_p\subset T_pM$, does not square to $0$, $I^2\ne 0$.   Letting $\mathcal W_I$ be the orthogonal complement of $\mathcal N$ in $\mathcal F_I$, and $\mathcal U_I$ the orthogonal complement of $\chi$ in $\mathcal V_I$ (the orthogonality is with respect to \eqref{inner_pr_V} and \eqref{in_prod_Q}), we obtain the splitting
\begin{equation}\label{split_TQ}
T_I\mathcal  Q\ =\ \overbrace{\mathcal H_I\oplus\underbrace{\mathcal U_I\oplus \mathbb R\cdot\chi_I}_{ \mathcal V_I}}^{ \mathcal D_I}\oplus\underbrace{\mathcal W_I\oplus\mathbb R\cdot\mathcal N_I}_{ \mathcal F_I}.
\end{equation}

We now consider a canonical 1-form $\eta$ on $\mathcal Q$, defined at any $I=x_sI_s\in \mathcal Q$,  by 
\begin{equation}\label{can_1_form}
\eta\  =\ \sum_s x_s\,\pi^\ast(\eta_s),
\end{equation}
where $\pi^\ast(\eta_s)$ is the pullback of $\eta_s$ via $\pi:\mathcal Q\rightarrow M$.  In order to calculate the exterior derivative of $\eta$, we introduce three local 1-forms $\phi_1,\phi_2$ and $\phi_3$ on $\mathcal Q$ by the formula
\begin{equation}\label{def_phi_s}
\phi_i=\epsilon_i\,dx_i-x_j\,\pi^\ast(\alpha_k)-x_k\,\pi^\ast(\alpha_j)
\end{equation} 
for any positive permutation $(ijk)$ of $1,2,3.$  Clearly, the forms $\phi_s$ are defined only within the coordinate chart \eqref{loc_coord}. According to Lemma~\eqref{lemma_lift}, each $\phi_s$ vanishes on the horizontal distribution $\mathcal D$ and we have
\begin{equation}\label{def_phi_s}
\phi_s\Big(\frac{\partial}{\partial x_t}\Big)\ =\ \begin{cases}
\epsilon_s,\ \text{if}\ s=t\\
0,\ \text{otherwise}.
\end{cases}
\end{equation}
For any $A\in T_I\mathcal Q$, we have \footnotemark
\begin{equation}\label{decomp_A}
A= \Big(\big(\pi_{\ast}A\big)_{H}\Big)^h\ +\ \sum_s\epsilon_s\Big(\eta_s(A)\,\xi_s^h\ +\ \phi_s(A)\,\frac{\partial}{\partial x_s}\Big).
\end{equation} 
\footnotetext{By subscript $H$ we mean projection onto $H$ w.r.t. the decomposition \eqref{TM_split}}
\begin{lemma}\label{deta_formula}
The exterior derivative of the canonical 1-form $\eta$ on $\mathcal Q$ is given  \big(within the coordinate chart \eqref{loc_coord}\big) by {\footnotemark}
\begin{equation*}
d\eta\ =\  {\sum_{(ijk)}}\Big(2x_i\,\pi^\ast(\omega_i)\ +\ \epsilon_i\,\phi_i\wedge\pi^\ast(\eta_i)\ -\ \frac{Scal}{8n(n+2)}\,\epsilon_i\,x_i\,\pi^\ast(\eta_j\wedge\eta_k)\Big).
\end{equation*} 
\footnotetext{The 2-forms $\omega_s$ are as in \eqref{def_omegas}; for the wedge product, we are using the formula\\ $\phi_i\wedge\pi^\ast(\eta_i)\,\big( A,B)=\phi_i(A)\,\eta_i(\pi_\ast B)-\phi_i(B)\,\eta_i(\pi_\ast A)$}
\end{lemma}
\begin{proof} Differentiating \eqref{can_1_form} yields
\begin{equation*}
d\eta\ = \sum_s \Big(dx_s\wedge \pi^\ast(\eta_s)\ +\ x_s\,\pi^\ast(d\eta_s)\Big).
\end{equation*}
We calculate at $I$ \big($(ijk)$ being any positive permutation of 1,2,3\big):
\begin{equation*}
\hspace{-4.4cm}d\eta(X^h,\tilde X^h)\ =\ \sum_s x_sd\eta_s(X,\tilde X)\ =\ 2\sum_s x_s g(I_sX,\tilde X);
\end{equation*}
\begin{equation*}
\hspace{-1.5cm}\begin{aligned}
d\eta\big(X^h,\,\xi^h_i\big)\ &=\ \epsilon_i\,dx_i(X^h)\ +\ \sum_s x_sd\eta_s(X,\xi_i)\\ 
&\overset{cf. \,\eqref{lift}}{=}\  x_j\alpha_k(X)-x_k\alpha_j(X)\ +\ \sum_s x_sd\eta_s(X,\xi_i) \ \overset{cf.\, \eqref{coneforms}}{=}\ 0;
\end{aligned}
\end{equation*}
\begin{equation*}
\begin{aligned}
d\eta\big(\xi^h_i,\,\xi^h_j\big)\ &=\ \epsilon_j\,dx_j(\xi_i^h)-\epsilon_i\,dx_i(\xi_j^h)+\sum_sx_sd\eta_s(\xi_i,\xi_j)\\
&\overset{cf. \,\eqref{lift}}{=}\  x_k\alpha_i(\xi_i)-x_i\alpha_k(\xi_i)-x_j\alpha_k(\xi_j)+x_k\alpha_j(\xi_j)+\sum_sx_sd\eta_s(\xi_i,\xi_j)\\
&\overset{cf.\, \eqref{coneforms}}{=}\ -\frac{Scal}{8n(n+2)}\,x_k;
\end{aligned}
\end{equation*}
\begin{equation*}
\hspace{-9.5cm}d\eta\Big(\xi^h_s,\,\frac{\partial}{\partial x_t}\Big)\ =\ -\eta_t(\xi_s).
\end{equation*}
\end{proof}

As a consequence of the previous lemma, we obtain the following. 

\begin{cor} \label{1-form_eta}
At any $I=\sum_sx_s\,I_s\in \mathcal Q$, the canonical 1-form $\eta$ and the vector field~$\chi$ $($cf. \eqref{def_chi_N}$)$ satisfy
\begin{equation*}
\eta (\chi)\ =\ \sum_s\epsilon_s\,x_s^2 \qquad \text{and\,}\footnotemark\qquad \chi\lrcorner d\eta\ =\ -\sum_s\epsilon_s\,x_s\,dx_s.
\end{equation*}
\end{cor}
\footnotetext{$\chi\lrcorner d\eta(A)=d\eta(\chi,A)$}

Let ${\mathcal Q}^o\subset\mathcal Q$ be the open subset consisting of all $I\in\mathcal Q$ with $I^2\ne 0.$ Clearly, the twistor and the reflector spaces $\mathcal Z$ and $\mathcal R$ are submanifolds in ${\mathcal Q}^o$.  On the manifold ${\mathcal Q}^o$, we have the distribution 
\begin{equation}\label{defKI}
\mathcal K=\mathcal H\oplus\mathcal U\oplus \mathcal W \subset  T{\mathcal Q}^o.
\end{equation}
If using the local coordinates \eqref{loc_coord} and the 1-forms $\phi_s$ (cf. \eqref{def_phi_s}), $\mathcal K$ can be described with the equations
\begin{equation*}
\sum_s\,x_s\,\pi^\ast(\eta_s)\ =\ 0\qquad \text{and}\qquad\sum_s\,x_s\,\phi_s\ =\ 0.
\end{equation*} 
We introduce a natural field $J$ of endomorphisms of the distribution $\mathcal K$
that satisfies $J^2=-\lc I,I\rc\,\text{id}$ by setting
\begin{equation}\label{def_J}
J(X+U+W)\ =\ \Big(I\big(\pi_\ast X\big)\Big)_I^h\ +\ \chi_I\times U\ +\ \mathcal N_I\times W,
\end{equation}
\noindent where $X\in \mathcal H_I$, $U\in \mathcal U_I$ and $W\in\mathcal W_I$.  For any $A\in \mathcal K_I$ within the coordinate chart \eqref{loc_coord}, we have (cf. \eqref{decomp_A})
\begin{equation}\label{def_J_end}
\begin{aligned}
J(A)\ =\ \sum_{s}x_s\,\Big(I_s\pi_\ast(A)_H\Big)^h\ +\sum_{(ijk)}&\Big(\epsilon_j\,x_j\,\eta_k(\pi_\ast A)-\epsilon_k\,x_k\,\eta_j(\pi_\ast A)\Big)\xi_i^h\\
& +\sum_{(ijk)}\Big(\epsilon_j\,x_j\,\phi_k(A)-\epsilon_k\,x_k\,\phi_j(A)\Big)\frac{\partial}{\partial x_i}.
\end{aligned}
\end{equation}

Let us denote by $G$ the bilinear form
\begin{equation*}
G(A,B)=-\frac{1}{2\lc I,I\rc}d\eta(JA,B),\qquad A,B\in \mathcal K_I.
\end{equation*} 
Since $J^2=-\big\lc I,I\big\rc\,\text{id}$, we  have
\begin{equation*}
d\eta(A,B)\ =\ 2G(JA,B),\qquad A,B\in \mathcal K_I.
\end{equation*}

\begin{lemma}\label{G_formula} At any $I\in {\mathcal Q}^o$,  $G$ is a symmetric 2-form on $\mathcal K_I\subset T_I{\mathcal Q}^o$ $\big($cf. \eqref{defKI}$\big)$ of signature $(2n+2,2n+2)$ that satisfies the relation
\begin{equation*}
 G(JA,B)=-G(A,JB)
\end{equation*}
for  $A,B\in\mathcal K_I$.  Explicitly, within the coordinate chart \eqref{loc_coord}, we have that
\begin{equation}\label{form_G}
\begin{aligned}
G(A,B)\ =\ &g\Big(\big(\pi_\ast A\big)_H,\big(\pi_\ast B\big)_H\Big)\\
& -\ \frac{Scal}{16n(n+2)}\,\sum_s\epsilon_s\, \eta_s(\pi_\ast A)\,\eta_s(\pi_\ast B)\\
&-\ \frac{1}{2\lc I,I\rc}\sum_{(ijk)}\,\epsilon_i\,x_i\,\Big(\phi_j(A)\,\eta_k(\pi_\ast B)\ +\ \eta_k(\pi_\ast A)\,\phi_j(B)\\
&\hspace{4cm}\ -\ \phi_k(A)\,\eta_j(\pi_\ast B)\ -\ \eta_j(\pi_\ast A)\,\phi_k(B)\Big).
\end{aligned}
\end{equation}
\end{lemma}

\begin{proof}
Formula \eqref{form_G} is a straightforward application of Lemma~\ref{deta_formula}. To calculate the signature of $G$ on $\mathcal K_I$ we first observe that the two subspaces $\mathcal H_I$ and $\mathcal U_I+\mathcal W_I$ are $G$-orthogonal and the restriction of $G$ to $\mathcal H_I$ has the same signature as $g$. Therefore, we only need to show that the restriction of $G$ to  $\mathcal U_I+\mathcal W_I$ is of signature $(2,2)$. 

For any fixed $I\in{\mathcal Q}^o$,  we can pick a local pqc-structure $(\eta_s,I_s,g)$ in such a way so that either $I=\lambda I_3$ or $I=\lambda I_1$, $\lambda\in \mathbb R$. In the first case, the restriction of $G$ to $\mathcal U_I+\mathcal W_I$ is given, w.r.t. the frame $\{\xi^h_1,\xi^h_2\}$ of $\mathcal U_I$ and the frame $\Big\{\frac{\partial}{\partial x_1}, \frac{\partial}{\partial x_2}\Big\}$ of $\mathcal W_I$, by the matrix
\begin{equation}\label{mat1_in_lemma_G}
\begin{pmatrix}
h & 0&  0&l\\
0 & h&  -l&0\\
0 & -l&  0&0\\
l & 0&  0&0\\
\end{pmatrix},
\end{equation} 
where $h=\frac{Scal}{16n(n+2)}$ and $l=\frac{1}{2\lambda}$. This matrix has two eigenvalues, each with multiplicity two:
$
\frac{1}{2}\Big(h\pm\sqrt{h^2+4l^2}\Big).
$
Therefore, the restriction of $G$ to $\mathcal U_I+\mathcal W_I$ has signature (2,2).

Similarly, in the second case (when $I=\lambda I_1$), the restriction of $G$ to $\mathcal U_I+\mathcal W_I$ is given, w.r.t. the frame $\{\xi^h_2,\xi^h_3\}$ of $\mathcal U_I$ and the frame $\Big\{\frac{\partial}{\partial x_2}, \frac{\partial}{\partial x_3}\Big\}$ of $\mathcal W_I$, by 
\begin{equation*}
\begin{pmatrix}
h & 0&  0&l\\
0 & -h&  -l&0\\
0 & -l&  0&0\\
l & 0&  0&0\\
\end{pmatrix}.
\end{equation*}
This matrix has two positive and two negative eigenvalues:
\begin{equation*}
\frac{1}{2}\Big(\pm h+\sqrt{h^2+4l^2}\Big)\qquad \text{and} \qquad \frac{1}{2}\Big(\pm h-\sqrt{h^2+4l^2}\Big),
\end{equation*}
and thus the signature is again $(2,2)$. 
\end{proof}

\subsection{Invariance}

For the definition of the distribution $\mathcal K\subset T{\mathcal Q}^o$ and the respective field $J$ \big(cf. \eqref{def_J}\big) we have used, as an essential tool, the concept of a horizontal lift of a vector fields w.r.t. $\nabla$. Since $\nabla$ is the canonical connection  determined by a choice of  a section $g$ of the canonical line bundle $\mathcal G\rightarrow M$  (cf. \S~\ref{subsection-pqc-str}), the whole construction depends on that choice as well. Our purpose here is to show that this dependence is only formal and in fact, if we replace  $g$ with 
\begin{equation}\label{bar_g}
\bar g=\frac{1}{2f}g,
\end{equation}
 where $f$ is any smooth and non-vanishing function on $M$, then both $\mathcal K$ and $J$ remain unchanged.

If $A$ is a vector field on $M$ with a horizontal lift $A^h$ to $\mathcal Q$ w.r.t. $g$ and $\nabla$, we shall denote by $A^{\bar h}$ the respective horizontal lift of $A$ to $\mathcal Q$ w.r.t. $\bar g$ and its canonical connection $\overline{\nabla}$.  Clearly, if $(\eta_s,I_s,g)$ is any local pqc structure for $H$, then so is $(\overline{\eta}_s,I_s,\bar g)$, where $\overline{\eta}_s=\frac{1}{2f}\eta_s$. More generally, we shall use the bar on objects related to the pqc structure $(\eta_s,I_s,g)$ to indicate the respective objects related to $(\overline{\eta}_s,I_s,\bar g)$, e.g. $\overline{\xi}_s$ will denote the Reeb vector fields \big(cf. \eqref{Reeb}\big), defined by
\begin{equation*}
\begin{aligned}
&(i) \quad\overline{\eta}_s(\overline{\xi}_t)=
\begin{cases}
\epsilon_s,\ \text{if}\ s=t \\
0,\ \text{otherwise}
\end{cases}\\
&(ii)\quad d\overline{\eta}_s(\overline{\xi}_t,X)+d\overline{\eta}_t(\overline{\xi}_s,X)=0,\quad \forall X \in H.
\end{aligned} 
\end{equation*}
One can easily derive from the above that 
\begin{equation*}
\overline{\xi}_s\ =\ 2f\,\xi_s\ +\ I_s\nabla f,
\end{equation*}
where $\nabla f$ is the horizontal gradient of the function $f$; that is, the unique section of the distribution $H$ that satisfies $g(\nabla f,X)=df(X)$ for all $X\in H$. According to \cite{CIZ2}\footnote{Here we are using slightly different sign conventions}, we have the following formulas concerning the connection 1-forms $\overline{\alpha}_s$ \big(cf. \eqref{nabla_cross_1}\big) of~$\overline{\nabla}$:
\begin{equation}\label{alpha_bar}
\begin{aligned}
& \overline{\alpha}_s(X)\ =\ \alpha_s(X)\ +\ \frac{\epsilon_s}{f}\,df(I_s\,X),\qquad s=1,2,3,\qquad \forall X \in H,\\
&\overline{\alpha}_i(\overline{\xi}_i)\ =\ 2f\,\alpha_i(\xi_i)\ -\ \frac{\Delta f}{2n}\ +\  \frac{g(\nabla f,\nabla f)}{n\,f}\\
&\overline{\alpha}_j(\overline{\xi}_i)\ =\ 2f\,\alpha_j(\xi_i)\ +\ \alpha_j(I_i\nabla f) \ -\ 2\epsilon_i\,df(\xi_k)\\
&\overline{\alpha}_k(\overline{\xi}_i)\ =\ 2f\,\alpha_k(\xi_i)\ +\ \alpha_k(I_i\nabla f) \ +\ 2\epsilon_i\,df(\xi_j),
\end{aligned} 
\end{equation}
where $(ijk)$ is any positive permutation of $1,2,3$, and $\Delta f=\sum_a\nabla df(e_a,e^*_a)$ (cf.~\eqref{dual_basis}).

\begin{lemma}

Within the coordinate chart $\eqref{loc_coord}$ on ${\mathcal Q}$, we have the following formulas for the horizontal lift of a vector fields from $M$ to $\mathcal Q$:  
\begin{equation}\label{bar_h}
\begin{aligned} 
X^{\bar h}\ =&\ X^h\ +\ \mathcal N\times\sum_t \epsilon_t \,df(I_tX)\,\frac{\partial}{\partial x_t},\\
{\overline\xi_s}^{\bar h}\ =&\ 2\,f\,\xi_s^h\ +\ \big(I_s\,\nabla f\big)^h\ +\ 2\,df\big(\pi_*\chi\big)\frac{\partial}{\partial x_s}\\ 
&-\ 2\epsilon_s\, x_s\,\sum_t \epsilon_t \,df(\xi_t)\,\frac{\partial}{\partial x_t}\
+\ \Big(-\ \frac{\Delta f}{2n}\ +\  \frac{g(\nabla f,\nabla f)}{n\,f}\Big)\ \mathcal N\times\frac{\partial}{\partial x_s},
\end{aligned} 
\end{equation}
where $X$ is any section of $H$ and  $s=1,2,3$.
\end{lemma}

\begin{proof} The proof is obtained by a straightforward calculation using  \eqref{alpha_bar} and \eqref{lift}.
\end{proof}

Let us observe that the vector field $\mathcal N$ on $\mathcal Q$  \big(cf. \eqref{def_chi_N}\big) does not depend on the choice of $g$ and $\nabla$, whereas the field $\chi$ is changed as follows:
\begin{equation*}
\begin{aligned} 
{\overline\chi}\ =\ &\sum_s\, x_s\,{\overline\xi_s}^{\bar h}\\ 
&=\ \chi\ +\ 2\,df(\pi_\ast\chi)\,\mathcal N\ +\ J\big(\nabla f\big)^h\ -\ 2\,\lc I,I\rc\,\sum_t \epsilon_t \,df(\xi_t)\,\frac{\partial}{\partial x_t}.
\end{aligned} 
\end{equation*}

\begin{prop}\label{inv_K_J} The distribution $\mathcal K$ on ${\mathcal Q}^o$ $\big($defined by \eqref{defKI}$\big) $and the field $J$ of endomorphisms of $\mathcal K$ $\big($defined by \eqref{def_J}$\big)$ do not depend on the choice of $g$ and $\nabla$.
\end{prop}

\begin{proof}
Let us begin by constructing a distribution $\overline{\mathcal K}$ on ${\mathcal Q}^o$ as in \eqref{defKI} by using \eqref{bar_g} in place of $g$ and $\overline\nabla$ in place of $\nabla$. Within the coordinate chart $\eqref{loc_coord}$, we have an orthogonal decomposition 
\begin{equation*}
\text{span}\Big\{\overline{\xi}_1^{\ \bar h},\overline{\xi}_2^{\ \bar h},\overline{\xi}_3^{\ \bar h}\Big\}\ = \ \overline{\mathcal U}\oplus\mathbb R.\overline\chi
\end{equation*}
that defines a distribution $\overline{\mathcal U}$. Then, 
\begin{equation*}
\overline{\mathcal K}\ =\ \overline{\mathcal H}\oplus\overline{\mathcal U}\oplus \mathcal W,
\end{equation*}
where the distribution $\overline{\mathcal H}$ is defined by the requirement that its sections are precisely the horizontal lifts, w.r.t. the connection $\overline\nabla$, of vector fields on $M$ tangent to the distribution $H\subset TM$, and $\mathcal W$ is as in~\eqref{split_TQ}.   

If $A=\sum_s\,a_s\overline{\xi}_s^{\ \bar h}$ is any element of $\overline{\mathcal U}_I$, $I=\sum_s x_sI_s$, then $\sum_s\, \epsilon_s\, a_s\,x_s=0$. Using \eqref{bar_h}, we calculate
\begin{equation*}
\begin{aligned} 
A\ =\ \sum_s\, &a_s\overline{\xi}_s^{\ \bar h}\ =\ \sum_s\bigg(2\,f\,a_s\,\xi_s^h\ +\ \big(a_s\,I_s\,\nabla f\big)^h\ +\ 2\,df\big(\pi_*\chi\big)\,a_s\,\frac{\partial}{\partial x_s}\\
&\hspace{3.5cm}+\ \Big(-\ \frac{\Delta f}{2n}\ +\  \frac{g(\nabla f,\nabla f)}{n\,f}\Big)\ \mathcal N\times \Big(a_s\,\frac{\partial}{\partial x_s}\Big)\bigg),
\end{aligned} 
\end{equation*}
which yields that $A\in \mathcal H\oplus\mathcal U\oplus\mathcal W$ and  $\overline{\mathcal U}\subset \mathcal K$. 
Similarly, if $X$ is any section of $H$ then, by \eqref{bar_h}, $X^{\bar h}\in \mathcal H\oplus\mathcal W$ and thus $\overline{\mathcal H}\subset K$. Therefore, we get that $\overline{\mathcal K}=\mathcal K$ and thus $\mathcal K$ does not depend on the choice of $g$ and $\nabla$.  The invariance of $J$ is shown similarly.

\end{proof}

\subsection{Proof of propositions \ref{prop_twistor} and \ref{prop_reflector}}

The restriction of the 1-form $\eta$ to the twistor space $\mathcal Z\subset {\mathcal Q}^o$ and the reflector space $\mathcal R\subset {\mathcal Q}^o$, respectively, satisfies
\begin{equation}
\eta\wedge d\eta^{2n}\ne 0
\end{equation} 
and therefore, it defines a contact structures on both $\mathcal Z$ and $\mathcal R$. The tangent bundles $T\mathcal Z$ and $T\mathcal R$, considered as subbundles in $T{\mathcal Q}^o$, are described by the equation
\begin{equation}
\sum_s\,x_s\,\phi_s\ =\ 0\qquad\text{cf. }\eqref{def_phi_s}.
\end{equation}
The vector field $\chi$ $\big($cf. \eqref{def_chi_N}$\big)$ is tangent to  $\mathcal Z$  (resp. $\mathcal R$), and if we restrict to the tangent space of $\mathcal Z$ (resp. $\mathcal R$), we obtain that  (cf. Corollary~\ref{1-form_eta})
\begin{equation*}
\eta (\chi)\ =\ \lc I,I\rc\qquad \text{and}\qquad \chi\lrcorner d\eta\ =0;
\end{equation*}
that is, $\chi$ is a Reeb vector field for the contact form $\eta$ on $\mathcal Z$ (resp. $\mathcal R$). 

At each $I\in\mathcal Z$ (resp. $I\in \mathcal R$), the 
kernel of $\eta$ $\big($cf. \eqref{defKI}$\big)$ is given  by the subspace $\mathcal K_I\subset T_I\mathcal Z$ (resp. $\mathcal K_I\subset T_I\mathcal R$) and the endomorphism $J$ $\big($cf. \eqref{def_J_end}$\big)$ of $\mathcal K_I$ satisfies $J^2=-id$ (resp. $J^2=id$). The pair $(\mathcal K,J)$ defines an almost CR structure on the twistor space $\mathcal Z$ and an almost para-CR structure on the reflector space $\mathcal R$. The signature of $d\eta(J.,.)$ is given by Lemma~\ref{G_formula}. By Proposition~\ref{inv_K_J}, the pair $(\mathcal K,J)$ is uniquely determined by the pqc-distribution $H\subset TM$ and does not depend on the particular choice of local pqc-structure $(\eta_s,I_s,g)$ for $H$.

\section{Integrability}\label{integrability}
In this section we consider the integrability question for the previously introduced (Section~\ref{TandRsp}) almost CR structure $(\mathcal K,  J)$ on the twistor space $\mathcal Z$,  and for the respective almost para-CR structure on the reflector space $\mathcal R$.

Observe that by Lemma~\ref{G_formula}, if $A$ and $B$ are any two sections of $\mathcal K$, then 
\begin{equation*}
 \big[JA,B\big]+\big[A,JB\big]
\end{equation*}  
is also a section of $\mathcal K.$
\noindent Therefore, the integrability  of the almost CR-structure $(\mathcal K, J)$ on $\mathcal Z$  is equivalent to the equation $N^{\mathcal Z}(A,B)=0$, where $N^{\mathcal Z}$ is the so called Nijenhuis tensor, defined by
\begin{equation}\label{def_NZ}
N^{\mathcal Z}(A,B)\ =\ -\big[A,B\big]\ +\ \big[JA,JB\big]\ -\  J\Big(\big[JA,B\big]+\big[A,JB\big]\Big),
\end{equation}
for any two vector fileds $A$ and $B$ on $\mathcal Z$ that are tangent to the distribution $\mathcal K\subset T\mathcal Z$.

The complexified distribution $\mathcal K^c=\mathcal K\otimes_{\mathbb R} \mathbb C$ on  $\mathcal Z$ splits as 
\begin{equation*}
\mathcal K^c=\mathcal K_{\sqrt{-1}}\ \oplus\ \mathcal K_{-\sqrt{-1}},
\end{equation*}
where $\mathcal K_{\sqrt{-1}}$ and $\mathcal K_{-\sqrt{-1}}$ are the eigenspaces of $J$ with eigenvalues $\sqrt{-1}$ and $-\sqrt{-1}$. The vanishing of the Nijenhuis tensor  $N^{\mathcal Z}$ is equivalent to the formal integrability of the complex distributions $\mathcal K_{\sqrt{-1}}$ and $\mathcal K_{-\sqrt{-1}}$; that is, to any of the following two conditions
\begin{equation*}
\Big[\mathcal K_{\sqrt{-1}},\ \mathcal K_{\sqrt{-1}}\Big]\subset \mathcal K_{\sqrt{-1}}\qquad \text{and}\qquad\Big[\mathcal K_{-\sqrt{-1}},\ \mathcal K_{-\sqrt{-1}}\Big]\subset \mathcal K_{-\sqrt{-1}}.
\end{equation*}

Similarly, the almost para-CR structure $(\mathcal K,  J)$ on the reflector space $\mathcal R$ is integrable if $N^{\mathcal R}(A,B)=0$ for any two sections $A$ and $B$ of the distribution $\mathcal K\subset T\mathcal R$, where
\begin{equation}\label{def_NR}
N^{\mathcal R}(A,B)\ =\ \big[A,B\big]\ +\ \big[JA,JB\big]\ -\  J\Big(\big[JA,B\big]+\big[A,JB\big]\Big).
\end{equation}
Here, the complexified distribution splits as $\mathcal K^c=\mathcal  K_{+1}\oplus\mathcal K_{-1}$,  where $\mathcal K_{+1}$ and $\mathcal K_{-1}$ are the $\pm 1$ eigenspaces of $J$. The vanishing of $N^\mathcal R$ is equivalent to the formal integrability of  $\mathcal K_{+1}$ and $\mathcal K_{-1}$, i.e., to the conditions
\begin{equation*}
\Big[\mathcal K_{+1},\ \mathcal K_{+1}\Big]\subset \mathcal K_{+1}\qquad \text{and}\qquad\Big[\mathcal K_{-1},\ \mathcal K_{-1}\Big]\subset \mathcal K_{-1}.
\end{equation*}

The following result is obtained as a straightforward application of Proposition~\ref{N_vanish} below.
\begin{prop}\label{prop_integ} The almost CR structure $(\mathcal K, J)$ on the twistor space $\mathcal Z$ and the respective almost para-CR structure on the reflector space $\mathcal R$ are integrable.  
\end{prop}

\subsection{Integrability on the ambient space ${\mathcal Q}^o$} The distribution $\mathcal K$ \big(cf.~\eqref{defKI}\big) can be considered as a vector bundle over the manifold ${\mathcal Q}^o$. We introduce a Nijenhuis-like tensor field  $N$ defined for any two vector fields $A$ and $B$ on ${\mathcal Q}^o$ that are tangent to the distribution $\mathcal K$ by the formula
\begin{equation}\label{def_N}
N(A,B)\ =\ -\lc I,I\rc\big[A,B\big]\ +\ \big[JA,JB\big]\ -\  J\Big(\big[JA,B\big]+\big[A,JB\big]\Big).
\end{equation}
 $N$ is indeed a tensor field---meaning that the value of $N(A,B)$ at any given $I\in{\mathcal Q}^o$ depends only on the values of $A$ and $B$ at $I$---due to the obvious property $N(fA,hB)=fhN(A,B)$ for any functions $f$ and $h$ on ${\mathcal Q}^o$.  Notice also that the expression on the right hand side in \eqref{def_N} makes sense, since, by Lemma~\ref{deta_formula}, the vector field $\big[JA,B\big]+\big[A,JB\big]$ is tangent to the distribution $\mathcal K$ and the action of $J$ is well defined there (by definition $J$ is a field of endomorphisms of $\mathcal K$, cf. \eqref{def_J}). Furthermore, applying Lemma~\ref{deta_formula} one more time, we observe that $N(A,B)$ is always a section of $\mathcal K$, and thus $N:\mathcal K\times\mathcal K\rightarrow \mathcal K$. Clearly, if restricting to the twistor space $\mathcal Z\subset {\mathcal Q}^o$, $N$ coincides with the Nijenhuis tensor $N^{\mathcal Z}$ $\big($cf. \eqref{def_NZ}$\big)$ and, similarly, on $\mathcal R \subset {\mathcal Q}^o$ it coincides with $N^{\mathcal R}$ $\big($cf. \eqref{def_NR}$\big)$.

\begin{prop}\label{N_vanish}
On ${\mathcal Q}^o$, we have that
\begin{equation*}
N(A,B)\ =\ 0,
\end{equation*}
for any two sections $A$ and $B$ of $\mathcal K.$
\end{prop}
\begin{proof} To begin with, we fix an arbitrary non-vanishing section $g$ of the line bundle $\mathcal G(M)\rightarrow M$ (cf. \S~\ref{subsection-pqc-str}) and consider the corresponding canonical connection  $\nabla$ on $TM$. By Lemma~\ref{lemma-pqc}, for any fixed $I\in{\mathcal Q}^o$, we can pick a local pqc-structure $(\eta_s, I_s,g)$ in such a way so that either $I=\lambda I_1$ or $I=\lambda I_3$, $\lambda\in\mathbb R$. Let us assume that $I=\lambda I_1$ (in the other case the proof is similar).
Using the corresponding Reeb vector fields $\xi_s$, we construct a coordinate chart as in \eqref{loc_coord} around the fixed point $I=\lambda I_1\in {\mathcal Q}^o$.

  Following the structure \eqref{defKI} of $\mathcal K$ and observing that $N(A,B)=-N(B,A)$, we see that there are six different cases to consider in the proof: (I) $A,B\in \mathcal H$; (II) $A\in \mathcal H$, $B\in \mathcal U$; (III) $A\in \mathcal H$, $B\in \mathcal W$; (IV) $A,B\in \mathcal U$; (V) $A\in\mathcal U$, $B\in \mathcal W$; and (VI) $A,B\in\mathcal W$.

\vspace{0.2cm}
Case (I) $A,B\in \mathcal H$:\nopagebreak
\par Without loss of generality, in this case, we may assume that $A=X^h$ and $B=Y^h$ for some vector fields $X$ and $Y$ on $M$ that are tangent to the distribution $H$. We calculate:

\begin{equation*}
\begin{aligned}
N(&X^h,Y^h)_{\big|I=\lambda I_1}\ =\\
& =\ \lambda^2\big[X^h,Y^h\big] \ +\ \sum_{s,t}\Big[x_s(I_sX)^h,\, x_t(I_tY)^h\Big]_{\big|I=\lambda I_1}\\
&\hspace{3cm}-\ \lambda\,I_1\sum_{s}\bigg(\Big[x_s(I_sX)^h,\,Y^h\Big]+\Big[X^h,\,x_s(I_sY)^h\Big]\bigg)_{\big|I=\lambda I_1}
\end{aligned}
\end{equation*}
\begin{equation*}
\begin{aligned}
\hspace{0.8cm}=\ \lambda^2\bigg(&\big[X,Y\big]+[I_1X, I_1Y\big]-I_1\Big(\big[I_1X,Y\big]+\big[X,\,I_1Y\big]\Big)\bigg)^h_{\big|I=\lambda I_1}\\
&+\ \sum_s\bigg(dx_s\big(JX^h\big)(I_sY)^h- dx_s\big(JY^h\big)(I_sX)^h\\
&\hspace{3cm} - dx_s(X^h)J(I_sY)^h+ dx_s(Y^h)J(I_sX)^h\bigg)_{\big|I=\lambda I_1}\\
&+\ 2\lambda^3\Big(\rho_3(X,Y)+\rho_3(I_1X,I_1Y)-\rho_2(I_1X,Y)-\rho_2(X,I_1Y)\Big)\,\frac{\partial}{\partial x_2}\\
&+\ 2\lambda^3\Big(\rho_2(X,Y)+\rho_2(I_1X,I_1Y)-\rho_3(I_1X,Y)-\rho_3(X,I_1Y)\Big)\,\frac{\partial}{\partial x_2}
\end{aligned}
\end{equation*}

Observe that the last two lines in the above expression vanish as a consequence  of~ \eqref{expr_for_rho}. We may represent the remaining part of the expression as 
\begin{equation}\label{NXY-main}
{\lambda^2\,\Big(\Sigma_1\Big)}^h_{\big|I=\lambda I_1}\ +\ {\lambda^2\,\Big(\Sigma_2\Big)}^h_{\big|I=\lambda I_1},
\end{equation} 
where
\begin{equation*}
\begin{aligned}
&\Sigma_1\ =\ \big[X,Y\big]+[I_1X, I_1Y\big]-I_1\Big(\big[I_1X,Y\big]+\big[X,\,I_1Y\big]\Big),\\
&\Sigma_2\ =\ \frac{1}{\lambda}\sum_s\bigg(dx_s\big((I_1X)^h\big)(I_sY)- dx_s\big((I_1Y)^h\big)(I_sX)\\
&\hspace{4.5cm} - dx_s(X^h)I_1(I_sY)+ dx_s(Y^h)I_1(I_sX)\bigg).
\end{aligned}
\end{equation*}

\noindent Using the canonical connection $\nabla$ on $M$ and its torsion $T$ (cf. \eqref{def_tor}), we calculate 
\begin{equation*}
\begin{aligned}
\Sigma_1 \ =\ &\nabla_XY-\nabla_YX-T(X,Y) + \nabla_{I_1X}(I_1Y)-\nabla_{I_1Y}(I_1X)-T(I_1X,I_1Y)\\
&-\ I_1\Big(\nabla_{I_1X}Y-\nabla_Y(I_1X)-T(I_1X,Y)+\nabla_{X}(I_1Y)-\nabla_{I_1Y}X-T(X,I_1Y)
\end{aligned}
\end{equation*}
\vspace{-0cm}
\begin{equation*}
\begin{aligned}
\hspace{0.8cm}&\ \ =\ -I_1\big(\nabla_XI_1\big)Y+I_1\big(\nabla_YI_1\big)X+\big(\nabla_{I_1X}I_1\big)Y-\big(\nabla_{I_1Y}I_1\big)X\\
&\hspace{3cm}-\  T(X,Y)-T(I_1X,I_1Y)+I_1\Big(T(I_1X,Y)+T(X,I_1Y)\Big)
\end{aligned}
\end{equation*}
\vspace{-0cm}
\begin{equation*}
\begin{aligned}
\hspace{0.8cm} =\ &\ \Big(\alpha_2(X)+\alpha_3(I_1X)\Big)I_2Y\ -\ \Big(\alpha_2(Y)+\alpha_3(I_1Y)\Big)I_2X\\
&\hspace{2cm}+\ \Big(\alpha_3(X)-\alpha_2(I_1X)\Big)I_3Y\ -\ \Big(\alpha_3(Y)-\alpha_2(I_1Y)\Big)I_3X,
\end{aligned}
\end{equation*}
\noindent where the last equality follows from \eqref{nabla_cross_1} and \eqref{torsion}. 

Applying \eqref{lift} to the expression $\Sigma_2$ gives
\begin{equation*}
\begin{aligned}
\Sigma_2 \ =\ -&\Big(\alpha_2(X)+\alpha_3(I_1X)\Big)I_2Y\ +\ \Big(\alpha_2(Y)+\alpha_3(I_1Y)\Big)I_2X\\
&\hspace{2cm}-\ \Big(\alpha_3(X)-\alpha_2(I_1X)\Big)I_3Y\ +\ \Big(\alpha_3(Y)-\alpha_2(I_1Y)\Big)I_3X
\end{aligned}
\end{equation*}
Therefore, by \eqref{NXY-main}, we get
$
N(X^h,Y^h)=0.
$

\vspace{0.2cm}
Case (II) $A\in \mathcal H$, $B\in \mathcal U$:\nopagebreak
\par Here, we may assume that $A=X^h$ and $B=\mu_2\,\xi_2^h+\mu_3\,\xi_3^h$, where $\mu_2$ and $\mu_3$ are any real numbers and  $X$ is a section of $H\subset TM$. We have that
\begin{equation}\label{xi_2_ext_0}
N(A,B)_{\big|I=\lambda I_1}\ = \ \mu_2\,N(X^h,\xi_2^h)_{\big|I=\lambda I_1}\ +\ \mu_3\,N(X^h,\xi_3^h)_{\big|I=\lambda I_1}
\end{equation}

In order to calculate the quantity  $N(X^h,\xi_2^h)_{\big|I=\lambda I_1}$, consider the vector field 
\begin{equation}\label{xi_2_ext}
\xi^h_2+\frac{x_2}{\lc I,I\rc}\chi.
\end{equation}
Clearly, \eqref{xi_2_ext} is a vector filed tangent to the distribution $\mathcal U\subset\mathcal K\subset T\mathcal Q$ (cf. \eqref{defKI}) that is defined in a neighborhood of the fixed point $I=\lambda I_1$, so that its value at this point coincides with the value of $\xi^h_2$. Therefore, we have that (cf.  \eqref{def_N} and \eqref{def_J})
\begin{equation*}
\begin{aligned}
N(X^h,\,&\xi_2^h)_{\big|I=\lambda I_1}\ =\ N\Big(X^h,\,\xi_2^h+\frac{x_2}{\lc I,I\rc}\chi\Big)_{\big|I=\lambda I_1}\\
&= \ \lambda^\pow{2}\big[X^h,\xi_2^h+\frac{x_2}{\lc I,I\rc}\chi\big]_{\big|I=\lambda I_1} \ +\ \sum_{s,t}\Big[x_s(I_sX)^h,\, x_t\big(\xi_t\times\xi_2\big)^h\Big]_{\big|I=\lambda I_1}\\
& -\ \lambda\,\sum_s\,J\Big(\big[x_s(I_sX)^h,\,\xi_2^h+\frac{x_2}{\lc I,I\rc}\chi\big]+\big[X^h,\, x_s(\xi_s\times\xi_2)^h\big]\Big)_{\big|I=\lambda I_1}
\end{aligned}
\end{equation*}
\begin{equation}\label{xi_2_ext_1}
\begin{aligned}
\hspace{0cm}& =\ \lambda^\pow{2}\big[X,\xi_2\big]^h \ +\ \lambda^\pow{2}\big[I_1X,\, \xi_3\big]^h\ -\ \lambda\,J\Big(\big[I_1X,\,\xi_2\big]+\big[X,\,\xi_3\big]\Big)^h\\
&\hspace{1.5cm} -\ \lambda^\pow{2}\,\bigg(\Big(-\alpha_2(\xi_2)+\alpha_3(\xi_3)\Big)I_2X\ +\ \Big(\alpha_2(\xi_3)-\alpha_3(\xi_2)\Big)I_3X\\
&\hspace{7.5cm}+\ \Big(\alpha_3(X)-\alpha_2(I_1X)\Big)\xi_1\bigg)^h\\
&\hspace{1.5cm} +\ 2\lambda^\pow{3}\Big(\rho_3(X,\xi_2)+\rho_3(I_1X,\xi_3)-\rho_2(I_1X,\xi_2)-\rho_2(X,\xi_3)\Big)\,\frac{\partial}{\partial x_2}\\
&\hspace{1.5cm} +\ 2\lambda^\pow{3}\Big(\rho_2(X,\xi_2)+\rho_2(I_1X,\xi_3)-\rho_3(I_1X,\xi_2)-\rho_3(X,\xi_3)\Big)\,\frac{\partial}{\partial x_3}
\end{aligned}
\end{equation}
By the properties \eqref{expr_for_rho} of $\rho(X,\xi_s)$, the last four lines in the above expression vanish.  Using the canonical connection $\nabla$ on $M$ and its torsion $T$ (cf. \eqref{def_tor}), we calculate that
\begin{equation*}
\begin{aligned}
\lambda^\pow{2}\big[X,&\xi_2\big]^h \ +\ \lambda^\pow{2}\big[I_1X,\, \xi_3\big]^h\ -\ \lambda\,J\Big(\big[I_1X,\,\xi_2\big]+\big[X,\,\xi_3\big]\Big)^h\\
&\ =\ \lambda^\pow{2}\,\bigg(I_1\big(\nabla_{\xi_2}I_1\big)X-\big(\nabla_{\xi_3}I_1\big)X-T(X,\xi_2)+I_1T(I_1X,\xi_2)\\
&\hspace{1.7cm}-T(I_1X,\xi_3)+I_1T(X,\xi_3)+\nabla_X\xi_2+\nabla_{I_1X}\xi_3\\
&\hspace{6cm}-\xi_1\times\Big(\nabla_{(I_1X)}\xi_2+\nabla_{X}\xi_3\Big)\bigg)^h
\end{aligned}
\end{equation*}
\begin{equation}\label{xi_2_ext_2}
\begin{aligned}
\hspace{0.7cm}&\ =\ \lambda^\pow{2}\,\bigg(\Big(-\alpha_2(\xi_2)+\alpha_3(\xi_3)\Big)I_2X\ +\ \Big(\alpha_2(\xi_3)-\alpha_3(\xi_2)\Big)I_3X\\
&\hspace{6.5cm}+\ \Big(\alpha_3(X)-\alpha_2(I_1X)\Big)\xi_1\bigg)^h,
\end{aligned}
\end{equation}
where for the last identity, we have used formulas \eqref{nabla_cross_1}, \eqref{nabla_cross_2} and \eqref{torsion}. Substituting \eqref{xi_2_ext_2} into \eqref{xi_2_ext_1}, we get 
\begin{equation*}
N(X^h,\,\xi_2^h)_{\big|I=\lambda I_1}=0.
\end{equation*}
Similarly, one can also show that
\begin{equation*}
N(X^h,\,\xi_3^h)_{\big|I=\lambda I_1}=0
\end{equation*}
and therefore, we have that in this case $N(A,B)=0$.

\vspace{0.2cm}
 Case (III) $A\in \mathcal H$, $B\in \mathcal W:$\nopagebreak
 \par We may assume here that $A=X^h$ and $B=\mu_2\,\frac{\partial}{\partial x_2}+\mu_3\,\frac{\partial}{\partial x_3}$, where $\mu_2$ and $\mu_3$ are any real numbers and  $X$ is a section of $H\subset TM$. Then,
\begin{equation}\label{xi_2_ext_0}
N(A,B)_{\big|I=\lambda I_1}\ = \ \mu_2\,N\Big(X^h,\frac{\partial}{\partial x_2}\Big)_{\big|I=\lambda I_1}\ +\ \mu_3\,N\Big(X^h,\frac{\partial}{\partial x_3}\Big)_{\big|I=\lambda I_1}
\end{equation}

In order to show that $N\Big(X^h,\frac{\partial}{\partial x_3}\Big)_{\big|I=\lambda I_1}$ vanishes (the vanishing of the other summand is shown similarly), we consider the vector field 
\begin{equation*}\label{partial_3_ext}
\frac{\partial}{\partial x_3}-\frac{x_3}{\lc I,I\rc}\mathcal N.
\end{equation*}
Clearly, this is a vector filed tangent to the distribution $\mathcal W\subset\mathcal K\subset T\mathcal Q$ (cf. \eqref{defKI}) that is defined in a neighborhood of the fixed point $I=\lambda I_1$, so that its value at this point coincides with the value of $\frac{\partial}{\partial x_3}$. Therefore, using \eqref{def_N} and \eqref{def_J} we get
\begin{equation*}
\begin{aligned}
N\Big(X^h,&\frac{\partial}{\partial x_3}\Big)_{\big|I=\lambda I_1}\ =\ N\Big(X^h,\frac{\partial}{\partial x_3}-\frac{x_3}{\lc I,I\rc}\mathcal N\Big)_{\big|I=\lambda I_1}\\
&= \ \lambda^\pow{2}\Big[X^h,\frac{\partial}{\partial x_3}-\frac{x_3}{\lc I,I\rc}\mathcal N\Big]_{\big|I=\lambda I_1} \ +\ \sum_{s,t}\Big[x_s(I_sX)^h,\, x_t\,\frac{\partial}{\partial x_t}\times\frac{\partial}{\partial x_2}\Big]_{\big|I=\lambda I_1}\\
& -\ \lambda\,\sum_s\,J\Big(\Big[x_s(I_sX)^h,\,\frac{\partial}{\partial x_3}-\frac{x_3}{\lc I,I\rc}\mathcal N\Big]+\Big[X^h,\, x_t\,\frac{\partial}{\partial x_t}\times\frac{\partial}{\partial x_2}\Big]\Big)_{\big|I=\lambda I_1}
\end{aligned}
\end{equation*}
\begin{equation*}\label{partial_2_ext_1}
\begin{aligned}
\hspace{0.7cm}& =\ \lambda^\pow{2}\Big[X^h,\,\frac{\partial}{\partial x_3}\Big] \ +\ \lambda^\pow{2}\Big[(I_1X)^h,\, \frac{\partial}{\partial x_2}\Big]\\
&\hspace{5.5cm} -\ \lambda\,J\bigg(\Big[(I_1X)^h,\, \frac{\partial}{\partial x_3}\Big]+\Big[X^h,\,\frac{\partial}{\partial x_2}\Big]\bigg)\\
& +\ \lambda^\pow{2}\,\Big(\alpha_2(X)-\alpha_3(I_1X)\Big)\frac{\partial}{\partial x_1}\ =\ 0.
\end{aligned}
\end{equation*}

\vspace{0.2cm}
 Case(IV) $A,B\in \mathcal U:$\nopagebreak
 \par It suffices to assume $A=\xi_2^h$, $B=\xi_3^h$. Using \eqref{def_N} and \eqref{lift}, we calculate 
 \begin{equation*}
\begin{aligned}
N\big(\xi_2^h,\,\xi_3^h&\big)_{\big|I=\lambda I_1}\ =\ N\Big(\xi_2^h+\frac{x_2}{\lc I,I\rc}\mathcal \chi,\, \xi_3^h-\frac{x_3}{\lc I,I\rc}\mathcal \chi\Big)_{\big|I=\lambda I_1}\\
&= \ \lambda^\pow{2}\Big[\xi_2^h+\frac{x_2}{\lc I,I\rc}\mathcal \chi,\, \xi_3^h-\frac{x_3}{\lc I,I\rc}\mathcal \chi\Big]_{\big|I=\lambda I_1}\\
&\hspace{3cm} \ +\ \sum_{s,t}\Big[x_s(\xi_s\times\xi_2)^h,\, x_t(\xi_t\times\xi_3)^h\Big]_{\big|I=\lambda I_1}\\
& -\ \lambda\,\sum_s\,J\bigg(\Big[x_s(\xi_s\times\xi_2)^h,\, \xi_3^h-\frac{x_3}{\lc I,I\rc}\mathcal \chi\Big]\\
&\hspace{3cm}+\ \Big[\xi_2^h+\frac{x_2}{\lc I,I\rc}\mathcal \chi,\, x_s(\xi_s\times\xi_3)^h\Big]\bigg)_{\big|I=\lambda I_1}\ =\ 0.
\end{aligned}
\end{equation*}

\vspace{0.2cm}
 Case(V) $A\in\mathcal U$, $B\in \mathcal W:$\nopagebreak
 \par Here, we need to consider the assumptions: $A=\xi_s^h$ and $B=\frac{\partial}{\partial x_t}$ for $s,t=2,3.$ We shall consider only the case where $s=2$ and $t=3$; the remaining three possibilities are entirely analogous. 
 \begin{equation*}
\begin{aligned}
N\Big(\xi_2^h,\,\frac{\partial}{\partial x_3}&\Big)_{\big|I=\lambda I_1}\ =\ N\Big(\xi_2^h+\frac{x_2}{\lc I,I\rc}\mathcal \chi,\, \frac{\partial}{\partial x_3}-\frac{x_3}{\lc I,I\rc}\mathcal N\Big)_{\big|I=\lambda I_1}\\
&= \ \lambda^\pow{2}\Big[\xi_2^h+\frac{x_2}{\lc I,I\rc}\mathcal \chi,\, \frac{\partial}{\partial x_3}-\frac{x_3}{\lc I,I\rc}\mathcal N\Big]_{\big|I=\lambda I_1}\\
&\hspace{3cm} \ +\ \sum_{s,t}\Big[x_s(\xi_s\times\xi_2)^h,\, x_t\,\frac{\partial}{\partial x_t}\times\frac{\partial}{\partial x_2}\Big]_{\big|I=\lambda I_1}\\
& -\ \lambda\,\sum_s\,J\bigg(\Big[x_s(\xi_s\times\xi_2)^h,\, \frac{\partial}{\partial x_3}-\frac{x_3}{\lc I,I\rc}\mathcal N\Big]\\
&\hspace{3cm}+\ \Big[\xi_2^h+\frac{x_2}{\lc I,I\rc}\mathcal \chi,\, x_t\frac{\partial}{\partial x_t}\times\frac{\partial}{\partial x_3}\Big]\bigg)_{\big|I=\lambda I_1}\ =\ 0.
\end{aligned}
\end{equation*}

\vspace{0.2cm}
 Case(VI) $A,B\in\mathcal W:$\nopagebreak
 \par It suffices to consider only the case $A=\frac{\partial}{\partial x_2}$, $B=\frac{\partial}{\partial x_3}$. 
 \begin{equation*}
\begin{aligned}
N\Big(\frac{\partial}{\partial x_2},\,&\frac{\partial}{\partial x_3}\Big)_{\big|I=\lambda I_1}\ =\ N\Big(\frac{\partial}{\partial x_2}+\frac{x_2}{\lc I,I\rc}\mathcal N,\, \frac{\partial}{\partial x_3}-\frac{x_3}{\lc I,I\rc}\mathcal N\Big)_{\big|I=\lambda I_1}\\
&= \ \lambda^\pow{2}\Big[\frac{\partial}{\partial x_2}+\frac{x_2}{\lc I,I\rc}\mathcal N,\, \frac{\partial}{\partial x_3}-\frac{x_3}{\lc I,I\rc}\mathcal N\Big]_{\big|I=\lambda I_1}\\
&\hspace{2cm} \ +\ \sum_{s,t}\Big[x_s\frac{\partial}{\partial x_s}\times \frac{\partial}{\partial x_2},\, x_t\frac{\partial}{\partial x_t}\times \frac{\partial}{\partial x_3}\Big]_{\big|I=\lambda I_1}\\
& -\ \lambda\,\sum_s\,J\bigg(\Big[x_s\frac{\partial}{\partial x_s}\times \frac{\partial}{\partial x_2},\, \frac{\partial}{\partial x_3}-\frac{x_3}{\lc I,I\rc}\mathcal N\Big]\\
&\hspace{2cm}+\ \Big[\frac{\partial}{\partial x_2}+\frac{x_2}{\lc I,I\rc}\mathcal N,\, x_s\frac{\partial}{\partial x_s}\times \frac{\partial}{\partial x_3}\Big]\bigg)_{\big|I=\lambda I_1}\ =\ 0.
\end{aligned}
\end{equation*}

\end{proof}


\begin{thebibliography}{}

\bibitem{AK0} Alekseevsky, D. \& Kamishima, Y., \emph{Quaternionic and para-quaternionic CR structure on
(4n+3)-dimensional manifolds}, CEJM 2(5) 2004 732-753.

\bibitem{AS} A. Andrada, S. Salamon, Complex product structures on Lie algebras, Forum Math. 17 (2) (2005) 261-295.

\bibitem{Biq} O. Biquard, M\'etriques d'Einstein asymptotiquement sym\'etriques, Ast\'erisque,
265 (2000).

\bibitem{Swann} A.S. Dancer,  H.R. Jorgensen, A.F.  Swann, \emph{Metric geometries over the split quaternions},
Rend. Sem. Mat. Univ. Politec. Torino 63 (2005), no. 2, 119-139.

\bibitem{DIM} J. Davidov, S. Ivanov, I. Minchev, The twistor space of a quaternionic contact manifold,  Quarterly Journal of Mathematics, 63, 873-890.

\bibitem{IMV1} S. Ivanov, I. Minchev, D. Vassilev, Quaternionic Contact Einstein Structures and the Quaternionic Contact Yamabe
Problem, in: Mem. of AMS, vol. 231, 2014, Number 1086.

\bibitem{IMV2} S. Ivanov, I. Minchev, D. Vassilev, Extremals for the Sobolev inequality on the seven dimensional quaternionic Heisenberg
group and the quaternionic contact Yamabe problem, J. Eur. Math. Soc. (JEMS) 12 (4) (2010) 1041-1067.


\bibitem{IMV3} S. Ivanov, I. Minchev, D. Vassilev, Solution of the qc Yamabe equation on a 3-Sasakian manifold and the quaternionic
Heisenberg group, Analysis \& PDE 16:3 (2023), 839-860
    https://msp.org/apde/2023/16-3/p07.xhtml/pc, DOI:10.2140/apde.2023..101 

\bibitem{IP} S. Ivanov, A. Petkov, The qc Yamabe problem on non-spherical quaternionic
contact manifolds, Journal de
Math\'ematiques Pures et Appliqu\'ees,  vol. 118, (2018), 44-81.  DOI: 10.1016/j.matpur.2018.06.011

\bibitem{CIZ2} S. Ivanov, M. Thomakova and S. Zamkovoy,  Conformal paraquaternionic contact curvature and the local flatness theorem, arXiv:2404.16703 [math.DG] 

\bibitem{IV} S. Ivanov, D. Vassilev, Extremals for the Sobolev Inequality and the Quaternionic Contact Yamabe Problem, World
Scientific Publishing Co. Pvt. Ltd., Hackensack, NJ, 2011.

\bibitem{IMZ} S. Ivanov, I. Minchev, S. Zamkovoy, Twistor and reflector spaces of almost para-quaternionic manifolds, Handbook of pseudo-Riemannian geometry and supersymmetry, IRMA Lectures in Mathematics and Theoretical Physics 16, 477-496 (2010).

\bibitem{L} P. Libermann, Sur le probleme d'equivalence de certains structures infinitesimales, Ann. Mat. Pura Appl. 36 (1954) 27-120.

\bibitem{Sal}S. M. Salamon, Quaternionic K\"ahler manifolds, Invent. Math. 67 (1982), 143--171.

\bibitem{CIZ} M. Tchomakova, S. Ivanov and S. Zamkovoy, Geometry of paraquaternionic contact structures, arXiv:2404.16713 [math.DG]  

\end{thebibliography}
\end{document}